\documentclass[12pt,sumlimits,intlimits]{amsart}
\usepackage{url,cite,hyperref,amsmath,amsthm,amssymb,mathtools, soul}
\usepackage[margin=1.25in]{geometry}
\usepackage{xcolor}
\usepackage{comment}

\newtheorem{theorem}{Theorem}[section]

\newtheorem{proposition}[theorem]{Proposition}
\newtheorem{corollary}[theorem]{Corollary}
\newtheorem{lemma}[theorem]{Lemma}
\theoremstyle{definition}
\newtheorem{definition}[theorem]{Definition}
\newtheorem{remark}[theorem]{Remark}

\newtheorem{question}[theorem]{Question}

\newcommand{\V}{\Vert}

\newcommand{\Z}{\mathbb{Z}}

\newcommand{\C}{\mathbb{C}}
\newcommand{\N}{\mathbb{N}}
\newcommand{\R}{\mathbb{R}}

\newcommand{\T}{\mathbb{T}}

\newcommand{\mh}{\mathcal{H}}

\newcommand{\id}{\textup{id}}
\newcommand{\dr}{\textup{dr}}

\newcommand{\nor}{\trianglelefteq}
\newcommand{\Stab}{\textup{Stab}}
\newcommand{\supp}{\textup{supp}}
\numberwithin{equation}{section}

\title[Finite decomposition rank for virtually nilpotent groups]{Finite decomposition rank  for virtually nilpotent groups}
\author{Caleb Eckhardt}
\email{eckharc@miamioh.edu}
\author{Elizabeth Gillaspy}
\email{elizabeth.gillaspy@mso.umt.edu}
\author{Paul McKenney}
\email{mckennp2@miamioh.edu}
\thanks{C.E.\ was partially supported by a grant from the Simons Foundation.  E.G.~was primarily supported by the Deutsches Forschungsgemeinschaft via SFB 878 (awarded to the Universit\"at M\"unster, Germany).  Part of this work was carried out while C.E. and E.G. were at the CRM in Bellaterra, Spain as part of the IRP ``Operator algebras: dynamics and
interactions'' in 2017.}
\address{%
  Department of Mathematics,
  Miami University,
  Oxford, OH, 45056,
  United States}
\address{%
 Department of Mathematical Sciences,
 University of Montana,
 32 Campus Drive \# 0864,
 Missoula, MT, 59812-0864, 
 United States
  }

\begin{document}
\maketitle
\begin{abstract} We show that inductive limits of virtually nilpotent groups have strongly quasidiagonal C*-algebras, extending results of the first author on solvable virtually nilpotent groups. We use this result to show that the decomposition rank of the group C*-algebra of a finitely generated virtually nilpotent group $G$ is bounded by $2\cdot h(G)!-1$, where $h(G)$ is the Hirsch length of $G.$ This extends and sharpens results of the first and third authors on finitely generated nilpotent groups.  It then follows that if a C*-algebra generated by an irreducible representation of a virtually nilpotent group satisfies the universal coefficient theorem, it is classified by its Elliott invariant.

\end{abstract}
\section{Introduction} 
Thanks to intense work by many mathematicians over the last thirty years on the Elliott classification program, we now know that for a large class of simple nuclear C*-algebras, the \emph{Elliott invariant} -- consisting of $K$-theoretic and tracial data -- constitutes a classifying invariant \cite{Elliott16,Gong15,Tikuisis17}.  In other words, two C*-algebras within this class which have the same Elliott invariant are isometrically $*$-isomorphic.  

Among finitely generated groups, the virtually nilpotent ones (Definition \ref{def:nilpotent}) turn out to be precisely those groups whose irreducible representations always generate simple nuclear C*-algebras \cite{Lance73, Poguntke81,Moore76}.  
Therefore irreducible representations of  virtually nilpotent groups form  the largest class of groups (and representations) to which one may hope to apply the results of the Elliott classification program.

In sharp contrast to the Lie group case, the C*-algebras generated by irreducible representations of discrete (non virtually abelian) nilpotent groups are extremely interesting and varied. See  \cite{Eckhardt16a, Elliott84,Packer87,Rieffel81,Rieffel88} for examples of such C*-algebras.  We were motivated in part by these two facts to undertake the current project; Theorem \ref{thm:main} brings us closer to our eventual goal of showing that the C*-algebras generated by irreducible representations of finitely generated, virtually nilpotent groups are classified by their Elliott invariants.

As is well-known (cf.~\cite{Elliott08} for more details), it is not possible to classify all simple nuclear C*-algebras by their Elliott invariants.  Nonetheless, decades of work, culminating in the recent papers \cite{Elliott16,Gong15,Tikuisis17}, has established  that for the class of unital, separable, simple, nuclear C*-algebras with finite nuclear dimension which satisfy the universal coefficient theorem,  the Elliott invariant is a classifying invariant. 

The main result of this paper (Theorem \ref{thm:main}) shows that the decomposition rank of $C^*(G)$ is finite for all finitely generated, virtually nilpotent groups $G$.  Since decomposition rank bounds nuclear dimension it follows that $C^*(G)$ also has finite nuclear dimension and therefore so do all of its quotients. Therefore if the C*-algebra generated by an irreducible representation of a virtually nilpotent group satisfies the universal coefficient theorem it is classifiable by its Elliott invariant.  Unfortunately, we suspect that verifying the universal coefficient theorem for these C*-algebras will be quite difficult; we discuss this further in Section \ref{sec:comments}.

In addition to its relevance to the Elliott classification program,  the decomposition rank (and  nuclear dimension) of a group C*-algebra has an intriguing relationship with the structure of the underlying group.  Although our Theorem \ref{thm:main} sheds more light on this relationship, it is not yet completely understood.  We summarize the known results here, and discuss some open questions in Section \ref{sec:comments} below.

When $G$ is abelian, the decomposition rank and nuclear dimension of $C^*(G) \cong C_0(\widehat{G})$ coincide with the topological dimension of $\widehat{G}$, and hence with the rank of $G$.  This also indicates why we restrict ourselves to  \emph{finitely generated} groups when trying to establish finite decomposition rank; the infinitely generated group $\Z^\N$, for example, has infinite (decomposition) rank.  
 To our knowledge, the first decomposition rank estimate for   non-abelian groups was Carri\'on's \cite[Theorem 2.2]{Carrion11}, which establishes that if a group $G$ is a central extension of a finitely generated abelian group by a finitely generated abelian group, then $C^*(G)$ has finite decomposition rank.
In \cite{Eckhardt15}, the first author showed that for a group $G$ of the form $\Z^d\rtimes \Z$, $G$ is virtually nilpotent if and only if  $C^*(G)$ is strongly quasidiagonal (see Section \ref{sec:drZstab} below).  Combined with \cite[Theorem 5.3]{Kirchberg04} and our Theorem \ref{thm:main} below, this implies that among groups of the form $\Z^d\rtimes \Z$, the  virtually nilpotent ones are precisely those whose C*-algebras have finite decomposition rank.  
In a similar vein, Hirshberg and Wu \cite{Hirshberg17} recently proved that \emph{all} groups of the form $\Z^d\rtimes \Z$ have finite nuclear dimension; thus, there exist (non virtually nilpotent) groups of this form which have infinite decomposition rank but finite nuclear dimension.
Finally, in \cite{Giol10}, Giol and Kerr provided examples of topological dynamical systems such that $C(X)\rtimes \Z$ has arbitrarily large radius of comparison.  One corollary of their work is that the C*-algebras of the finitely generated wreath products $\Z^d\wr\Z$ have infinite nuclear dimension. 

In summation, among finitely generated groups, 
our Theorem \ref{thm:main} shows that all virtually nilpotent groups have C*-algebras with finite decomposition rank. Although there exist groups $G$ such that $C^*(G)$ has infinite decomposition rank but finite nuclear dimension, and there are examples of  groups $G$ where $C^*(G)$ has infinite nuclear dimension, there are no known examples of non virtually nilpotent groups for which $C^*(G)$  has finite decomposition rank (see Section \ref{sec:comments} for more on this).

We close the introduction by describing the contents of this paper.  
In \cite{Eckhardt14b} the first and third authors proved that if $G$ is a finitely generated, nilpotent group then $C^*(G)$ has finite nuclear dimension.\footnote{In fact, the authors actually showed the stronger claim that $C^*(G)$ has finite \emph{decomposition rank} but failed to realize it--see Section \ref{sec:finitedrnilpotent} below.} The method of proof worked inductively on the Hirsch length of the group by decomposing $C^*(G)$ as a continuous field over the dual of its center and analyzing the nuclear dimension of the simple fibers.  

The analysis of the simple fibers was possible by recent breakthroughs in the classification theory of nuclear C*-algebras (see Theorem \ref{thm:manyequivalence}) and certain properties enjoyed by C*-algebras generated by irreducible representations of nilpotent groups (see Theorem \ref{thm:T1}).  Since C*-algebras generated by \emph{virtually} nilpotent groups enjoy many of the same properties as those generated by nilpotent groups, it seemed likely that this same sort of analysis could be carried out with virtually nilpotent groups.

One key difference between the nilpotent and virtually nilpotent cases  is the structure of the continuous fields defined by the group C*-algebras.  In the case of nilpotent groups, one decomposes $C^*(G)$ as a continuous field over the dual of its center $Z(G)$, and the fibers are twisted group C*-algebras $C^*(G/Z(G),\sigma)$ where $\sigma$ is a 2-cocycle with values in $\T$ (see Section \ref{sec:ctsfield}). 

On the other hand, a virtually nilpotent group may have trivial center. Therefore in Section \ref{sec:ctsfield} we show how to decompose $C^*(G)$ as a continuous field over a topological quotient $\widetilde{Z}$ of the dual of $Z(N)$, for a finite index nilpotent subgroup  $N$ of $G$.  In this case the fibers may no longer be twisted group C*-algebras, but they are twisted crossed products in the sense of Packer and Raeburn \cite{Packer89, Packer90, Packer92}. These twisted crossed products are closely tied to primitive quotients of both $C^*(G)$ and $C^*(N)$; see Proposition \ref{prop:Greenimp} and Corollary \ref{cor:Greenimprep} below.

The main roadblock to extending the results of \cite{Eckhardt14b} to the virtually nilpotent case was the absence of (a proof of) quasidiagonality.   In \cite{Eckhardt14} the first author showed that every representation of a solvable virtually nilpotent group is quasidiagonal.    In Section \ref{sec:strongqd} below, we take a much less technical -- and much smarter -- approach to quasidiagonality  than the one of \cite[Section 3]{Eckhardt14} to extend the quasidiagonality result of \cite{Eckhardt14} to cover all virtually nilpotent groups.  Our results in this section, as well as in Section \ref{sec:finitedr}, rely on an analysis of the relationship between quotients of $C^*(G)$ and quotients of $C^*(N)$ for finite index subgroups $N$  of $G$. 
While this material (presented in Section \ref{sec:GrelN} below) is straightforward, we hope that other researchers will  find the interplay between traces, representations, subgroups, and induced ideals fruitful and interesting.

With the quasidiagonality question settled, we turn our focus towards the analysis of the twisted crossed product fibers in the continuous field decomposition of $C^*(G).$   In Section \ref{sec:finitedrnilpotent} we refine the arguments from \cite{Eckhardt14b} to establish that $C^*(G)$ has finite decomposition rank when $G$ is finitely generated and nilpotent. This slightly altered approach carries the benefit of a better decomposition rank bound than the one obtained in \cite{Eckhardt14b}.  This groundwork is built upon in Section \ref{sec:finitedr} to bound the decomposition rank of the fiber algebras of $C^*(G)$, using the work of Matui and Sato on strongly outer (twisted) actions from \cite{Matui14a}.  Finally, we use Carri\'{o}n's lemma on decomposition rank of continuous fields \cite{Carrion11} to obtain our bound on the decomposition rank of $C^*(G)$.

\section{Preliminaries} \label{sec:prelim}
\subsection{Notation and Definitions} \label{sec:notation}
In this section, we recall the definitions of our main objects of study in this paper and set our notation.
\begin{definition} \label{def:nilpotent} For a group $G$, write $Z(G)$ for the center of $G.$ Set $Z_1(G)=Z(G)$ and  recursively define $Z_i(G)$ for $i\geq2$ by $Z_i(G)/Z_{i-1}(G)=Z(G/Z_{i-1}(G)).$  We say $G$ is  \textbf{nilpotent} if there is an $n\geq1$ such that $G=Z_n(G).$  The group $G$ is  \textbf{virtually nilpotent} if it contains a finite index nilpotent subgroup.

\end{definition}
 We  refer the reader to Segal's text \cite{Segal83} for more information on nilpotent groups.
\subsubsection{}Let  $G$ be a discrete group. We always denote by $e\in G$ the trivial element of $G.$ Let $C^*(G)$ denote the full group C*-algebra of $G.$ 
{The unitary generators of $C^*(G)$ arising from $G$ will be denoted by  $\{u_g\}_{g\in G}$.}
We let $\lambda$ denote the left regular representation of $G$ on $\ell^2G$ and write $\{ \delta_g:g\in G \}$ for the standard basis of $\ell^2G.$  The reduced group C*-algebra $C^*_r(G)$ is the sub-C*-algebra of $B(\ell^2G)$ generated by $\lambda(G)$.
 
  For amenable groups the canonical homomorphism from $C^*(G)$ onto $C^*_r(G)$ is an isomorphism.  Since virtually nilpotent groups are built up by extensions of abelian and finite groups they are (elementary) amenable groups. We therefore make no distinction between $C^*(G)$ and $C^*_r(G)$ when $G$ is virtually nilpotent.  We refer the reader to Chapter 7 of Pedersen's book \cite{Pedersen79} for more information.

\subsubsection{}For a *-representation $\pi$ of $A$ we write $H_\pi$ for the Hilbert space on which $\pi(A)$ acts. If $\phi$ is a state on $A$ we write $(\pi_\phi,L^2(A,\phi))$ for the GNS representation of $A$ associated with $\phi.$  If $\phi$ is a normalized positive definite function on $G$, then we write $(\pi_\phi,L^2(G,\phi))$ for the associated GNS representation of $G$ (equivalently, of $C^*(G)$). If $\pi$ is a unitary representation of $G$ we write $C^*_\pi(G)$ for the C*-algebra generated by $\pi(G).$
\subsubsection{} As with the correspondence between states on $C^*(G)$ and normalized positive definite functions on $G$, there is a correspondence between tracial states on $C^*(G)$ and normalized positive definite functions on $G$ that are constant on conjugacy classes.  Thus, we will use the word \textbf{trace} to indicate either a tracial state on a C*-algebra, or a normalized positive definite function on a group $G$ which is constant on conjugacy classes. The {\bf canonical trace} $\tau_G$ on a group $G$ is the function $\tau_G(g) = \delta_{e, g}$.
The set of all traces on a C*-algebra is convex, so we will use the phrase \textbf{extreme trace}  to indicate an extreme point in this convex set.  In many of the articles we reference, the term \emph{character} is used for an extreme trace.  In this paper we prefer to reserve the term \textbf{character} for a continuous homomorphism from an abelian group to $\T.$ 

\subsubsection{} \label{sec:trivialext} Let $G$ be a group and $\gamma:Z(G)\rightarrow\T$ be a character.  We denote by $\overline{\gamma}:G\rightarrow \C$ the trace on $G$ defined by $\overline{\gamma}(g)=\gamma(g)$ if $g\in Z(G)$ and $\overline{\gamma}(g)=0$ if $g\not\in Z(G);$ we call this the {\bf trivial extension} of $\gamma$.

\subsection{Representation theory facts} \label{sec:repfacts} We record some classical facts about representations of C*-algebras and virtually nilpotent groups that we will use repeatedly in this paper. Recall that a representation $\pi$ of a C*-algebra $A$ is a {\bf factor representation} if its double commutant $\pi(A)''$ is a factor: that is, a  von Neumann algebra with trivial center.
\begin{theorem} \label{thm:repfacts} Let $A$ be a unital separable C*-algebra and $\tau$ a trace on $A.$  Then the following are equivalent by \textup{\cite[Corollary 6.8.6]{Dixmier77}}:
\begin{enumerate}
\item $\tau$ is an extreme trace.
\item $\pi_\tau$ is a factor representation.
\end{enumerate}
Moreover if either condition (1) or (2) holds then $\ker(\pi_\tau)$ is a primitive ideal of $A$ by \textup{\cite[II.6.1.11, II.6.5.15]{Blackadar06}}.
\end{theorem}
In general  the  primitivity of $\ker(\pi_\tau)$ does not imply that $\tau$ is extreme\footnote{For example let $A$ be simple with more than one trace. Then any non-extreme trace produces the primitive ideal $\{ 0 \}.$}. However in the case of \emph{finitely generated} virtually nilpotent groups the situation is better.

\begin{theorem}[Poguntke, Howe, Moore and Rosenberg, Kaniuth] \label{thm:T1}  Let $G$ be a virtually nilpotent group.  Then every primitive ideal of $C^*(G)$ is maximal.  If $G$ is finitely generated, then  for every primitive ideal $I$ of $C^*(G)$ there is a unique extreme trace $\tau$ on $G$ such that
\begin{equation*}
I=\{ a\in C^*(G):\tau(a^*a)=0 \}.
\end{equation*}
Therefore for any trace $\tau$ on a finitely generated virtually nilpotent group $G$, the following statements are equivalent:
\begin{enumerate}
\item $\tau$ is an extreme trace.
\item $\pi_\tau$ is a factor representation.
\item $\ker(\pi_\tau)$ is a primitive ideal of $C^*(G)$.
\item $\ker(\pi_\tau)$ is a maximal ideal of $C^*(G)$.
\end{enumerate}
\end{theorem}
\begin{proof} Moore and Rosenberg showed in \cite{Moore76} that every primitive ideal is maximal for solvable virtually nilpotent groups and Poguntke extended this to all virtually nilpotent groups in \cite{Poguntke81}.  Howe first showed the uniqueness of trace for finitely generated torsion free nilpotent groups in \cite{Howe77} and Kaniuth extended this to all finitely generated virtually nilpotent groups in \cite[Theorem 2]{Kaniuth80}. To complete the equivalences, by Theorem \ref{thm:repfacts} we only need to show the primitivity of $\ker(\pi_\tau)$  implies that $\tau$ is extreme.  

Suppose that $\ker(\pi_\tau)$ is primitive and $\tau=t\tau_1+(1-t)\tau_2$ for some traces $\tau_1$ and $\tau_2$ and $0<t<1.$  Since $\tau(a^*a)=0$ implies that $\tau_1(a^*a)=\tau_2(a^*a)=0$ and $\ker(\pi_\tau)$ is maximal it follows that $\ker(\pi_\tau)=\ker(\pi_{\tau_1})=\ker(\pi_{\tau_2}).$   Therefore $\tau_1$ and $\tau_2$ both induce traces on $C^*_{\pi_\tau}(G).$  By uniqueness of trace we must have   $\tau=\tau_1=\tau_2.$
\end{proof}

\subsection{Twisted crossed products}
\label{sec:twisted-xprod}
In our analysis of the simple quotients of virtually nilpotent group C*-algebras, we use both Green's twisted covariance algebras \cite{Green80} and the twisted crossed products of Packer and Raeburn \cite{Packer89, Packer90, Packer92}.  

\begin{definition}(\cite[Definitions 2.1-2.4]{Packer89})
\label{def:twisted-xprod}
A {\bf twisted dynamical system} $(A, G, \alpha, \omega)$ consists of a C*-algebra $A$, a locally compact group $G$, and maps $\alpha: G \to \text{Aut}\,A$,  $\omega: G\times G \to UM(A)$, such that:
\begin{itemize}
\item The maps $\omega$ and (for fixed $a \in A$) $s \mapsto \alpha_s(a)$ are Borel measurable;
\item For any $s \in G$, $\omega(s, e ) = \omega(e, s) = 1$ and $\alpha_e = id$;
\item For any $s, t \in G$, $\alpha_s \circ \alpha_t = \text{Ad}\,\omega(s, t)\circ \alpha_{st}$;
\item For any $s, t, r \in G, $ $\alpha_r(\omega(s,t)) \omega(r, st) = \omega(r,s)\omega(rs, t) $.
\end{itemize}
A {\bf covariant representation} $(\pi, U)$ of a twisted dynamical system is a nondegenerate $*$-representation $\pi: A \to B(H_\pi)$ and a map $U: G \to U(H_\pi)$ such that 
\[ \pi(\alpha_s(a)) = \text{Ad}\, U_s(a), \qquad U_s U_t = \pi(\omega(s,t)) U_{st}\]
for all $a \in A$ and $s, t \in G$.
The {\bf twisted crossed product} $A \rtimes_{\alpha, \omega} G$ of a twisted dynamical system $(A, G, \alpha, \omega)$ is the universal object for covariant representations of $(A, G, \alpha, \omega)$.
\end{definition}

\begin{definition}(\cite[p.~196]{Green80})
\label{defn:green-twist}
A {\bf twisted covariant system} $(G, A, \mathcal T)$ consists of a C*-algebra $A$, an action $\alpha$ of $G$ on $A$, and a continuous homomorphism $\mathcal T: N \to M(A)$ for some closed normal subgroup $N$ of $G$, such that 
\[ \mathcal T(n) a \mathcal T(n)^{-1} = \alpha_n(a)  \text{ and } \mathcal T(s n s^{-1}) = \alpha_s(\mathcal T(n)), \quad \text{ for all } n \in N,  s \in G, a \in A.\]
\end{definition}

Proposition 5.1 of \cite{Packer89} shows that the associated {\bf twisted covariance algebra} $C^*(G,A, \mathcal T)$, defined and studied in \cite{Green80}, can also be realized as a twisted crossed product $A \rtimes_{\beta, \omega} G/N$, where (if $c: G/N \to G$ is a Borel cross-section)
\[ \beta_s = \alpha_{c(s)} \quad \text{ and } \quad \omega(s,t) = \mathcal{T}(c(s) c(t) c(st)^{-1}).\]

\subsection{Quasidiagonality, decomposition rank and $\mathcal{Z}$-stability} \label{sec:drZstab}

\subsubsection{Quasidiagonality} Our interaction with quasidiagonality will be minimal.  We only use easy-to-prove permanence properties and the quasidiagonality results from \cite{Eckhardt14}.  We state the definition and refer the interested reader to Brown and Ozawa \cite{Brown08} for more information. 

Given a separable Hilbert space $\mh$, we say that a subset $\mathcal S \subseteq B(\mh)$  is {\bf quasidiagonal}  if there is a sequence of finite rank orthogonal projections $(P_n)_{n\in \N} \subseteq B(\mh)$ such that $\lim_{n \to \infty} P_n(\xi) = \xi$  for all $\xi \in \mh$ and
\[\lim_{n \to \infty} \|P_nT-TP_n\|=0, \text{ for all } T\in \mathcal S.\] 
A representation $(\pi,\mh)$ of a C*-algebra $A$ is quasidiagonal if $\pi(A)$ is a quasidiagonal set of operators. A C*-algebra is quasidiagonal if it admits a faithful quasidiagonal representation. A C*-algebra  $A$ is {\bf strongly quasidiagonal} \cite{Hadwin87} if every representation of $A$ is quasidiagonal.

\subsubsection{Decomposition rank} Like its cousin (nuclear dimension), decomposition rank is a C*-algebraic generalization of topological covering dimension.
Even though the main result of this paper proves that certain C*-algebras have finite decomposition rank, we do not actually work with the definition  of decomposition rank.  Instead we prove structure theorems about our algebras so that decomposition rank is easy to deduce from known results.  Therefore we elect not to recall the lengthy definition and instead refer the reader to Kirchberg and Winter's paper on decomposition rank \cite{Kirchberg04} as well as Winter and Zacharias's paper on nuclear dimension \cite{Winter10} for more details. The relevant details for this paper are as follows.

Let $A$ be a  C*-algebras and let $\dr(A)$ denote the \textbf{decomposition rank} of $A$ as defined in \cite{Kirchberg04}. Let $J$ be an ideal of $A.$  Then by \cite[(3.3)]{Kirchberg04} we obtain
\begin{equation}
\dr(A/J)\leq \dr(A). \label{eq:qdr}
\end{equation}
For a collection of C*-algebras $A_1,...,A_n$ we have by \cite[(3.1)]{Kirchberg04}
\begin{equation}
\dr(\oplus_{i=1}^n A_i)=\max\{ \dr(A_i):1\leq i\leq n \}. \label{eq:drsum}
\end{equation}
For any $n\geq1$ we have by \cite[Corollary 3.9]{Kirchberg04}
\begin{equation}
\dr(M_n\otimes A)=\dr(A). \label{eq:drstable}
\end{equation}

\subsubsection{$\mathcal{Z}$-stability} Another important regularity property fundamental to our investigations is $\mathcal{Z}$-stability. The Jiang-Su algebra $\mathcal Z$ was introduced in \cite{Jiang99}; a C*-algebra $A$ is called $\mathcal{Z}$-\textbf{stable} if it absorbs the Jiang-Su algebra $\mathcal{Z}$ tensorially, i.e. $A\otimes \mathcal{Z}\cong A.$  As with decomposition rank, in this paper we do not actually work with the structure of $\mathcal{Z}$.  Instead we prove structure theorems that make $\mathcal{Z}$-stability easy to deduce from known results.  We refer the reader to the survey article \cite{Elliott08} for more information about $\mathcal{Z}$ and its role in the Elliott classification program.

In Theorem \ref{thm:mainirr}, we prove $\mathcal{Z}$-stability of certain twisted crossed product C*-algebras by an argument that is routine to those working in the classification theory of nuclear C*-algebras. We recall the relevant details.
\begin{proposition}[\textup{\cite[Proposition 7.2.2]{Rordam02}}] \label{prop:intertwine} Let $A$ be a separable C*-algebra and $\mathcal{U}$ be a non-principal ultrafilter on $\N.$ Let $A^\mathcal{U}$ be the C*-ultrapower of $A$ 
and consider $A\hookrightarrow A^\mathcal{U}$ as the diagonal embedding.  Then $A\otimes \mathcal{Z}\cong A$  if and only if there is an embedding of $\mathcal{Z}$ into $A^\mathcal{U}\cap A'.$
\end{proposition}
As an immediate corollary we obtain the form of the argument that we will use in this paper. 
\begin{proposition} \label{prop:Zstablefixedpoint} Let $A$ be a separable $\mathcal{Z}$-stable C*-algebra. Let $G$ be a countable amenable group with a twisted action $(\alpha,\omega)$ on $A.$  Let $\mathcal{U}$ be a non-principal ultrafilter on 
$\N.$ The twisted action of $G$ on $A$ induces a canonical diagonal twisted action of $G$ on $A^\mathcal{U}.$  Let $A^\mathcal{U}_\alpha=\{ a\in A^\mathcal{U}: \alpha_s(a)=a, \textrm{ for all }s\in G \}.$ If there is an embedding of $\mathcal{Z}$ into $A^\mathcal{U}_\alpha\cap A'$, then the twisted crossed product $A\rtimes_{\alpha,\omega} G$ is $\mathcal{Z}$-stable. In particular, if each $\alpha_s$ defines an inner automorphism of $A,$ then $A\rtimes_{\alpha,\omega}G$ is $\mathcal{Z}$-stable.
\end{proposition}
\begin{proof} To see the first implication one notices that $A_\alpha^\mathcal{U}\cap A'\subseteq (A\rtimes_{\alpha,\omega}G)^\mathcal{U}\cap (A\rtimes_{\alpha,\omega}G)'.$ The conclusion then follows from Proposition \ref{prop:intertwine}.  For the second implication, observe that $A^\mathcal{U}\cap A'\subseteq A_\alpha^\mathcal{U}\cap A'$ in this case, and apply  the first claim.
\end{proof}
The following theorem is a culmination of years of work by many mathematicians.  We detail attribution in the proof for the interested reader.
\begin{theorem}[\cite{Bosa16}] \label{thm:manyequivalence}  Let $A$ be a simple, infinite-dimensional, quasidiagonal, nuclear separable unital C*-algebra with a unique trace.  Then the following are equivalent:
\begin{enumerate}
\item $\dr(A)\leq 1$.
\item $A$ has finite nuclear dimension.
\item $A$ is $\mathcal{Z}$-stable.
\end{enumerate}
\end{theorem}
\begin{proof}  We always have $\dr(A) \geq \text{dim}_{\text{nuc}}(A)$, so the first implication is trivial.  Winter proved that (2) implies (3) in \cite{Winter12} without the extra assumptions in our statement.  Matui and Sato proved in \cite[Theorem 1.1]{Matui14} that $\mathcal{Z}$-stability, together with the hypotheses of the current theorem, implies $\dr(A) \leq 3$.  Finally, Bosa et al.~built on the work of  Matui and Sato to prove, in \cite[Theorem 7.5]{Bosa16}, that (3) implies (1) under  hypotheses on the tracial state space that include the unique trace case.
\end{proof}

Combining the referenced results of Sections \ref{sec:repfacts} and \ref{sec:drZstab} with \cite{Eckhardt14b} we obtain
\begin{theorem}\label{thm:nildr1} Let $G$ be a finitely generated nilpotent group and $\pi$ an irreducible representation of $G$.  Then $\dr(C^*_\pi(G))\leq 1.$
\end{theorem}
\begin{proof} By Theorem \ref{thm:T1}, the C*-algebra $C^*_\pi(G)$ is simple and has a unique trace.  By \cite{Eckhardt14} it is quasidiagonal and by \cite{Eckhardt14b} it has finite nuclear dimension.  Therefore by Theorem \ref{thm:manyequivalence} we obtain $\dr(C^*_\pi(G))\leq 1.$
\end{proof}

\section{The structure of $C^*(G)$}
\label{sec:structure}
There are several key steps in our  proof of  Theorem \ref{thm:main} below, that $C^*(G)$ has finite decomposition rank  whenever $G$ is virtually nilpotent.  
First, we write $C^*(G)$ as a continuous field whose fiber algebras can be described in two ways: (i) as twisted crossed products of the form $C^*_\pi(N)\rtimes_{\alpha,\omega}G/N$, for certain representations $\pi$ of a normal subgroup $N\nor G$ (Proposition \ref{prop:cts-field}); and (ii) as C*-algebras of the form  $C^*_\pi(H) \otimes M_{H/N}$ for some irreducible representation $\pi$ of $G$ and certain finite index subgroups $H, N$ of $G$ (Proposition \ref{prop:Greenimp}).
Second, a careful analysis of the relationship between irreducible representations of $C^*(G)$ and those of $C^*(N)$, for $N \nor G$ a finite index subgroup, enables us to show $C^*_\pi(G)$ is quasidiagonal for every irreducible representation $\pi$ of $G$.  Third, the dual nature of the fibers described above in (i) and (ii) allow us to use quasidiagonality and results of Matui and Sato to establish $\mathcal{Z}$-stability of the simple fibers of $C^*(G)$ and therefore bound their decomposition rank by 1. Finally, Section \ref{sec:finitedrnilpotent} establishes that $\dr(C^*(N)) < \infty$ when $N$ is nilpotent; then, arguments similar to those employed in \cite{Eckhardt14b} finish the proof of Theorem \ref{thm:main}.
Thus, we begin with this structural analysis of $C^*(G)$.
\subsection{Continuous field structure} \label{sec:ctsfield}  
Recall (cf.~\cite[IV.1.6.]{Blackadar06} or \cite[Appendix C]{Williams08}) that  a C*-algebra $A$ is a \textbf{continuous field} over a topological space $X$ if there is a continuous open map $\Psi: \text{Prim}\,A \to X$.  By the Dauns-Hofmann Theorem, such a map induces an embedding $\Phi: C_0(X) \to ZM(A) \cong C^b(\text{Prim}\, A)$. The {\bf fiber algebra} $A_x$ of $A$ at $x \in X$ is $A/I_x$, where 
\[ I_x = \overline{\text{span}} \{\Phi( f) a: f \in C_0(X \backslash x), a \in A\}.\]

We recall some  classical results about group $C^*$-algebras and the various flavors of twisted crossed product $C^*$-algebras. 
Let $G$ be a countable discrete group with $N\nor G$, and fix a choice $c: G/N \to G$ of coset representatives with $c(e) = e$.
The Corollary to \cite[Proposition 1]{Green80} tells us that $C^*(G)$ can be viewed as the $C^*$-algebra of the (trivial) twisted covariant system
$(G, C^*(N), \iota)$, where  $G$ acts on $C^*(N)$ by conjugation and $\iota: N \to C^*(N)$ is the usual embedding. (If $N$ is amenable, as will be the case in our applications below, $\iota$ takes $n\in N$ to the associated generator $u_n$ of $C^*(N)$.)
Thus,  \cite[Proposition 5.1]{Packer89} establishes that $C^*(G)$ can also be understood as a twisted crossed product (Definition \ref{def:twisted-xprod}):
\[ C^*(G) \cong C^*(G, C^*(N), \iota) \cong C^*(N) \rtimes_{ \alpha, \omega} G/N,\]
where the action $\alpha$ of $G/N$ and the 2-cocycle $\omega: G/N \times G/N \to C^*(N)$ are given by
\begin{equation}
\alpha_s (a) = \iota({c(s)}) \, a \,\iota(c(s)^{-1}), \qquad \omega(s, t) = \iota(c(s) c(t)  c(st)^{-1}).\label{eq:omegadef}
\end{equation}
This picture of $C^*(G)$ as a twisted crossed product is fundamental to our understanding of the structure of its primitive quotients.

The universal property of $C^*(G, C^*(N), \iota)$ (as established in  \cite[Section 5]{Packer89}) implies that the isomorphism class of the twisted crossed product is independent of $c$.

By  \cite[Theorem 1.2]{Packer92}, $C^*(N)$ is a continuous field over $\widehat{Z(N)}$.  The fiber algebra at $\gamma \in \widehat{Z(N)}$ is the twisted group $C^*$-algebra $C^*(N/Z(N), \sigma_\gamma)$, where 
\[ \sigma_\gamma([n], [m]) = \gamma(\iota (\tilde n \tilde m \widetilde{nm}^{-1})).\]
Here $\tilde n \in N$ indicates a lift of $[n] \in N/Z(N)$.
Observe also (see the remarks at the beginning of  \cite[Section 4]{Eckhardt16}) that 
$C^*(N/Z(N), \sigma_\gamma) $ is isomorphic to $C^*_{\pi_{\overline \gamma}}(N)$, the GNS representation of $N$ with respect to the trace $\overline \gamma$ given by extending $\gamma \in \widehat{Z(N)}$ to a trace on $N$ which is zero outside of $Z(N)$.

\begin{definition}
Let $\widetilde Z$ denote the quotient of $\widehat{Z(N)}$ by the induced action $\hat \alpha$ of $G/N$: $\hat \alpha_s(\gamma)(z) = \gamma(\alpha_s(z))$. For $\gamma\in \widehat{Z(N)}$ we let $[\gamma]$ denote its equivalence class in $\widetilde Z.$
\label{def:cts-field-base}
\end{definition}
 The following proposition is a straightforward consequence of Nilsen's work \cite{Nilsen96} on twisted crossed products of continuous fields.

\begin{proposition}
\label{prop:cts-field}
Let $G$ be a discrete group, $N \nor G$, and $\widetilde{Z}$ be the quotient of $\widehat{Z(N)}$ from Definition \ref{def:cts-field-base}.  Then $C^*(G)$ is a continuous field over $\widetilde{Z}$.  The fiber over $[\gamma] \in \widetilde{Z}$ is given by 
\[C^*(G)_{[\gamma]} \cong C^*(N)_{[\gamma]} \rtimes_{\tilde \alpha, \tilde{\omega}} G/N,\]
where $C^*(N)_{[\gamma]} \cong \bigoplus_{\eta \in [\gamma]} C^*(N/Z(N), \sigma_\eta) \cong \bigoplus_{\eta\in [\gamma]} C^*_{\pi_{\overline \eta}}(N)$, and (denoting by $[u_n]_\eta$ the image of $n \in N$ in $C^*(N/Z(N), \sigma_\eta) \cong  C^*_{\pi_{\overline \eta}}(N)$)
\begin{equation}
\tilde{\alpha}_r([u_n]_\eta) = [ u_{c( r) n c( r)^{-1}}]_{\hat \alpha_r(\eta)}; \qquad \tilde \omega( r, s)= \bigoplus_{\eta \in [\gamma]} [\omega(r,s)]_{\eta}
\label{eq:cts-field-action}
\end{equation}
for a choice $c: G/N \to G$ of coset representatives.
\end{proposition}
\begin{proof}
The map $\widehat{Z(N)} \to \widetilde Z$ is a quotient map, hence continuous and open.  Consequently, $C^*(N)$ also has the structure of a continuous field over $\widetilde Z$.  In fact, the fiber algebra of $C^*(N)$ over $[\gamma] \in \widetilde Z$ is $C^*(N)/I_{[\gamma]}$, where
\[ I_\eta= \overline{\text{span}} \{ f \cdot a: f \in C_0(\widehat{Z(N)} \backslash \eta), a \in C^*(N)\} \subseteq C^*(N), \quad \text{ and } \quad I_{[\gamma]} = \prod_{s \in G/N} I_{\hat \alpha_s(\gamma)}.\]  For $\gamma \ne \eta$ with $[\gamma] = [\eta]$, the ideals $I_\gamma, I_\eta$
are coprime because $\widehat{Z(N)}$ is compact and Hausdorff, and hence admits a partition of unity.  In particular, given any two different points $\gamma, \eta \in \widehat{Z(N)}$, there exist $f, g \in C(\widehat{Z(N)})$ with $f(\gamma) = 0, g(\eta) = 0$, and $f+g =1$.  Thus, 
 the Chinese remainder theorem implies
 that the fiber algebra $C^*(N)_{[\gamma]}$ over $[\gamma] \in \widetilde Z$ is 
\[ C^*(N)_{[\gamma]} = \bigoplus_{\eta \in [\gamma]} C^*(N/Z(N), \sigma_{\eta}) = \bigoplus_{\eta \in [\gamma]} C^*_{\pi_{\overline{\eta}}}(N).\] 

 The induced action of $G/N$ on ${\widetilde{Z}}$ is trivial by construction, and hence (by  \cite[Corollary 5.3]{Nilsen96}) $C^*(G)$ is also a continuous field over ${\widetilde{Z}}$, with fiber algebra 
 \[C^*(G)_{[\gamma]} \cong C^*(N)_{[\gamma]} \rtimes_{\tilde \alpha, \tilde{\omega}} G/N.\] 
 The formulas for $\tilde \alpha, \tilde \omega$ given in \eqref{eq:cts-field-action} are direct translations of the formulas given in  \cite[Theorem 5.1]{Nilsen96}.
In other words, in the fibers of the bundle $C^*(G)$ over $\widetilde Z$, 
the twisted action permutes the summands within each fiber, and also acts on each summand $C^*(N/Z(N), \sigma_{\hat \alpha_s(\gamma)})$ by conjugation.
\end{proof}
The stabilizer group $H = \Stab_G(\gamma)$ of $\gamma \in \widehat{Z(N)}$, and  \cite[Theorem 2.13.(i)]{Green80}, provide another picture of the fiber algebras $C^*(G)_{[\gamma]}$.
\begin{proposition} \label{prop:Greenimp} Let $G$ be a discrete group and $N \nor G$ a finite index normal subgroup.  Fix $\gamma \in \widehat{Z(N)}$ and let $H=\Stab_G(\gamma)$ under the action described in Definition \ref{def:cts-field-base}.  Then
\begin{equation*}
C^*(G)_{[\gamma]}\cong (C^*(N/Z(N),\sigma_\gamma)\rtimes _{\widetilde{\alpha}, \widetilde{\omega}} H/N)\otimes M_{G/H} \cong C^*_{\pi_{\overline \gamma}}(H) \otimes M_{G/H}.
\end{equation*}
\end{proposition}
\begin{proof} 
Recall from  \cite[Proposition C.5]{Williams08} that since $C^*(N)$ is a continuous field over $\widehat{Z(N)}$, every irreducible representation $\pi$ of $C^*(N)$ factors through an irreducible representation of $C^*(N)_\eta$ for exactly one $\eta \in \widehat{Z(N)}$.  Thus, we can define a 
$G$-equivariant map $\psi: \text{Prim}\, C^*(N)_{[\gamma]} \to G/H$ by $\psi(P) = gH, $ if $P$ is the kernel of an irreducible representation of $C^*(N)_{g \cdot \gamma}$.

By construction of $C^*(N)_{[\gamma]}$, the map $\psi$ is onto.  Moreover, $\psi^{-1}(eN)$ consists precisely of the primitive ideals $P$ of $C^*(N)_{[\gamma]}$ such that
\[P \cap C^*(N)_\eta = C^*(N)_\eta \text{ for all } \eta \not= \gamma.\]
Setting $I = \bigcap \{ P \in \psi^{-1}(eN)\}$ we see that 
\[C^*(N)_{[\gamma]}/I \cong C^*(N)_\gamma \cong  C^*_{\pi_{\overline \gamma}}(N)\cong C^*(N/Z(N), \sigma_\gamma).\]

We now use \cite[Proposition 5.1]{Packer89} to observe that $C^*(G)_{[\gamma]}$ can be realized as the twisted covariance algebra $C^*(G, C^*(N)_{[\gamma]}, \mathcal{T})$, where $\mathcal{T}: N \to C^*(N)_{[\gamma]}$ is given by 
\[ \mathcal{T}(n) = \bigoplus_{\eta \in [\gamma]} [u_n]_\eta,\]
and the action of $G$ on $C^*(N)_{[\gamma]}$ is given by $g \cdot [u_n]_\eta = [u_{gng^{-1}}]_\eta$.

   Theorem 2.13(i) of \cite{Green80} now implies that 
   \[C^*(G)_{[\gamma]} \cong C^*(H, C^*(N/Z(N), \sigma_\gamma), \mathcal{T}^{A/I}) \otimes M_{G/H},\]
    where  $A := C^*(N)_{[\gamma]}$ and $\mathcal{T}^{A/I}(n) = [u_n]_\gamma$.
Another application of \cite[Proposition 5.1]{Packer89} produces 
\[C^*(H, C^*(N/Z(N), \sigma_\gamma), \mathcal{T}^{A/I}) \cong C^*(N/Z(N), \sigma_\gamma) \rtimes_{\widetilde{\alpha}, \widetilde{\omega}} H/N;\]
the final isomorphism $C^*(H, C^*(N/Z(N), \sigma_\gamma), \mathcal{T}^{A/I}) \cong C^*_{\pi_{\overline{\gamma}}}(H)$ follows from  \cite[Proposition 1]{Green80} by observing that since $H = \Stab_G(\gamma)$, the extension $\overline{\gamma}$ of $\gamma$ to a function on $H$ is again a trace.
\end{proof}

\subsection{An explicit embedding} \label{sec:embedding}  
Let $G$ be a discrete group and $N\nor G$  a finite index normal subgroup.  It is well known  that a choice of $G/N$-coset representatives determines an explicit embedding of $C^*_r(G)$ into $M_{G/N}\otimes C^*_r(N);$  we present the details here in order to lay the groundwork for our analysis of  the relationship between primitive ideals of $G$ and those of $N$.  

Let $e\in F\subseteq G$ be a complete set of $G/N$ coset representatives. For each $g\in G$ we uniquely decompose $g=c(g)n(g)$ where $c(g)\in F$ and $n(g)\in N.$
Using the fact that $c(gx)^{-1} gx = n(gx)$, 
we obtain a representation of $G$ into $M_{G/N}\otimes C^*_r(N)$ given by
\begin{equation} \label{eq:Gembed}
\rho(g)= \sum_{x\in F}e_{gxN,xN}\otimes \lambda_{n(gx)}
\end{equation}
where $e_{gxN,xN}$ is the matrix unit mapping $\delta_{xN}$ to $\delta_{gxN}$. 

Let $\tau_G,\tau_N$ be the canonical traces on $G$ and $N.$   Let $\tau_{G/N}$ be the unique trace on $M_{G/N}.$  It is easy to see that $\tau_{G/N}\otimes \tau_N\circ \rho(g)=\tau_G(g)$ for all $g\in G.$  Therefore the representation \eqref{eq:Gembed} extends to an embedding of $C^*_r(G)$ into $M_{G/N}\otimes C^*_r(N).$ 

Important to Section \ref{sec:strongqd} is the restriction of this embedding  to $C^*_r(N):$  
\begin{equation} \label{eq:Nembed}
\rho(h) = \sum_{x\in F}e_{hxN,xN}\otimes \lambda_{n(hx)}=\sum_{x\in F}e_{xN,xN}\otimes \lambda_{x^{-1}hx}.
\end{equation}

\subsection{Basics of the relationship between ideals of $C^*(G)$ and $C^*(N)$} \label{sec:GrelN}

 Let $G$ be a countable virtually nilpotent group and $\pi$ an irreducible representation of $G$. We now lay the groundwork for the proofs of strong quasidiagonality (Section \ref{sec:strongqd}) and finite decomposition rank  (Section \ref{sec:finitedr}) of $C^*_\pi(G)$ by analyzing the behavior of $\pi$ under restriction to  a finite index normal subgroup $N\nor G$.
 In this section, we do not require $G$ to be finitely generated. We prove Theorem \ref{thm:mainirr} via induction on the order of the group $|G/N|$ -- for this reason we do not require that $N$ is nilpotent in this section.  
 
We will use the setup and notation of the following definition throughout the paper.

\begin{definition}[The fundamental situation] \label{def:setup} Let $G$ be a virtually nilpotent discrete group and $N\nor G$ a  finite index subgroup.  Let $(\pi,H_\pi)$ be an irreducible representation of $G.$  Fix a complete set $e\in F\subseteq G$  of $G/N$ coset representatives.  As in Section \ref{sec:embedding}, the choice $F$ determines an embedding of $C^*(G)$ into $M_{G/N}\otimes C^*(N)$ provided explicitly by Equation (\ref{eq:Gembed}).   By \cite[Proposition 4.1.8]{Pedersen79}  we can extend $\pi$ to an irreducible representation $\id_{M_{G/N}}\otimes \sigma$ of $M_{G/N}\otimes C^*(N)$, i.e.~we embed  $H_\pi\subseteq H_{\id_{M_{G/N}}\otimes \sigma}$ such that if $P$ is the projection onto $H_\pi$ we have
\begin{equation}
\pi(a)=P(\id_{M_{G/N}}\otimes \sigma (a))P\quad \textup{ for all }a\in C^*(G). \label{eq:cutdown}
\end{equation}
For each $x\in F$, the fact that $N \nor G$ implies that the representation 
\begin{equation}
\sigma_x(h) : =\sigma(xhx^{-1}) \label{eq: sigmaxdef}
\end{equation}
is irreducible.
Then $G$ acts transitively on the set of ideals $\{ \ker(\sigma_x)\subseteq C^*(N):x\in F \}$ by conjugation.
Let $\Stab_G(\ker(\sigma))$ be the stabilizer of $\ker(\sigma)$ under this action.  Since $\ker(\sigma)$ is an ideal of $C^*(N)$ we clearly have $N\nor \Stab_G(\ker(\sigma)).$
Let $e\in F'\subseteq F$ be a choice of left $G/\Stab_G(\ker(\sigma))$ coset representatives.   Define
\begin{equation}
J=\bigcap_{x\in F'}\ker(\sigma_x). \label{eq:Jdef}
\end{equation}
\end{definition}
The following is just a recasting of a part of  Proposition \ref{prop:cts-field} in the representation theoretic language necessary for Section \ref{sec:strongqd}.
\begin{lemma} \label{lem:directsum} Assume the conditions of Definition \ref{def:setup}. Then\begin{equation}
C^*(N)/J\cong \bigoplus_{x\in F'} C^*(N)/\ker(\sigma_x). \label{eq:Ndsum}
\end{equation}
\end{lemma}
\begin{proof} Since $N$ is virtually nilpotent and each representation $\sigma_x$ is irreducible, each ideal $\ker(\sigma_x)$ is maximal by Theorem \ref{thm:T1}.  Therefore Equation (\ref{eq:Ndsum}) follows from the Chinese remainder theorem.
\end{proof}
\begin{lemma} \label{lem:Jkerpi} Assume the conditions of Definition \ref{def:setup}. We have $J=\ker(\pi)\cap C^*(N).$
\end{lemma}
\begin{proof} Let $h\in C^*(N).$ By Equations  (\ref{eq:Nembed}),  (\ref{eq:cutdown}) and  (\ref{eq: sigmaxdef})  we have
\begin{equation*}
\pi(h)=P(  \bigoplus_{x\in F}\sigma_x(h) )P
\end{equation*}
Therefore $J\subseteq \ker(\pi)\cap C^*(N)$ by (\ref{eq:Jdef}). 
It therefore follows from Lemma \ref{lem:directsum}  that there is a finite subset $F''\subseteq F'$ such that $\ker(\pi)\cap C^*(N)=\cap_{x\in F''}\ker(\sigma_x).$ Since $\ker(\pi)\cap C^*(N)$ is invariant under conjugation by elements of $G$, and $G$ acts transitively on the set of ideals $\{ \ker(\sigma_x):x\in F \}$  we must have $\cap_{x\in F''}\ker(\sigma_x)=\cap_{x\in F'}\ker(\sigma_x)$, that is,  $J=\ker(\pi)\cap C^*(N).$ 
\end{proof}

We now elaborate on the setting of Definition \ref{def:setup}.  Our goal is to understand the ideal $J_G$ of $C^*(G)$ induced by $J$; we will not have $J_G = \ker(\pi)$ in general, but the interplay between $J, J_G$, and $\ker(\pi)$ is fundamental to our proofs in later sections of strong quasidiagonality and finite decomposition rank for $C^*(G)$.  A first step in this direction is Corollary \ref{cor:Greenimprep} below.
\begin{definition} \label{def:tau} Assume the conditions of Definition \ref{def:setup}. Let $G$ act by conjugation on the state space of $C_\pi^*(G).$  Since $G$ is amenable,
there is a fixed point under this action, i.e. a trace on $C_\pi^*(G).$  Call this trace $\tau.$ Notice that $\tau$ also canonically defines a trace on $C^*(G)$ (and therefore also on $C^*(N)$) and we will use the same letter to denote this trace.  We remark that if $G$ is finitely generated, then $\tau$ is unique by Theorem \ref{thm:T1}.

Since $C^*_\pi(G)$ is simple, $\tau$ defines a faithful trace on $C^*_\pi(G).$  Therefore
\begin{equation*}
\ker(\pi)=\{ a\in C^*(G):\tau(a^*a)=0 \}.
\end{equation*}
From this equation and Lemma \ref{lem:Jkerpi} we have
\begin{equation*}
J=\ker(\pi)\cap C^*(N)=\{ a\in C^*(N):\tau(a^*a)=0  \}.
\end{equation*}
Recall the  embedding $C^*(G)\hookrightarrow M_{G/N}\otimes C^*(N)$ defined by Equation (\ref{eq:Gembed}).  Let $\tau_{G/N}$ be the unique trace on $M_{G/N}$ and define a trace on $G$ by
\begin{equation*}
\phi = (\tau_{G/N}\otimes (\tau|_N))|_G. \label{eq:phidef}
\end{equation*}
Let $J_G$ be the ideal of $C^*(G)$ generated by $J:$ 
\begin{equation*}
J_G=\left\{  \sum_{x\in F} a_x \lambda_x: a_x\in J  \right\}.
\end{equation*}
Lemma \ref{lem:Jkerpi} implies that $J_G \subseteq \ker (\pi)$.  
By the embedding of Equation (\ref{eq:Gembed}) and the definition of $\tau_{G/N}$, for any $a_x\in C^*(N)$ we have
\begin{equation}
\phi\left( \sum_{x\in F} a_x \lambda_x   \right)=\phi(a_e)=\tau(a_e). \label{eq:phitau}
\end{equation}
Let $\{a_x\}_{x\in F}\subseteq C^*(N)$ and set $a=\sum a_x\lambda_x\in C^*(G).$  The invariance of $\tau$ under conjugation now implies that
\begin{align*}
\phi(a^*a)&=\phi\left(  \sum_{x\in F}\lambda_{x^{-1}}a_x^*a_x\lambda_x  \right) \quad \textrm{by (\ref{eq:phitau})}\\
& = \tau \left( \sum_{x \in F} \lambda_{x^{-1}} a_x^* a_x \lambda_x \right) = \tau\left(  \sum_{x\in F}a_x^*a_x \right).
\end{align*}
Therefore $\phi(a^*a)=0$ if and only if $\tau(a_x^*a_x)=0$ for all $x\in F$, if and only if $a_x\in J$ for all $x\in F$, if and only if $a\in J_G.$ Collapsing these statements we obtain
\begin{equation}
J_G=\{ a\in C^*(G):\phi(a^*a)=0 \}.\label{eq:JGphi}
\end{equation}
\end{definition}

We thus obtain a representation-theoretic analogue of Proposition \ref{prop:Greenimp}.  We will eventually find both formulations useful.
\begin{corollary} \label{cor:Greenimprep} Assume the conditions of Definition \ref{def:setup}. Let $H=\Stab_G(\ker(\sigma)).$ Clearly $N\nor H\leq G$ so $H$ is finite index in $G.$  Then
\begin{equation*}
C^*(G)/J_G\cong C^*_{\pi_\phi}(G)\cong C^*_{\pi_\phi|_{H}}(H)\otimes M_{G/H}. \
\end{equation*}
\end{corollary}
\begin{proof} The first isomorphism is Equation (\ref{eq:JGphi}); the second is proved analogously to Proposition \ref{prop:Greenimp}. First, note that for any $x \in F$, the automorphism $\alpha_{xN}$ of $C^*(N)$ from Equation \eqref{eq:cts-field-action},
\[ \alpha_{xN}(a) = \lambda_x a \lambda_{x^{-1}},\]
preserves $J = \ker( \pi_\phi) \cap C^*(N) = \{ a \in C^*(N): \tau(a^*a) = 0\}$.  Thus, the twisted action $(\alpha, \omega)$ of $G/N$ on $C^*(N)$ of  Equation \eqref{eq:cts-field-action} descends to a twisted action of $G/N$ on $C^*_{\pi_\tau}(N)$.  Moreover, the correspondence between twisted crossed products and twisted covariant systems established in \cite[Proposition 5.1]{Packer89} implies that
\[ C^*_{\pi_\tau}(N) \rtimes_{\alpha, \omega} G/N \cong C^*(G, C^*_{\pi_\tau}(N), \iota),\]
and the universal properties of the twisted crossed product and of $C^*(G)$, combined with the fact that $\ker (\pi_{\tau|_N}) = \ker(\pi_\phi) \cap C^*(N)$, establish that 
\[C^*_{\pi_\tau}(N) \rtimes_{\alpha, \omega} G/N  \cong C^*_{\pi_\phi}(G).\]

Lemma \ref{lem:directsum} implies that $C^*(N)/J \cong C^*_{\pi_\tau}(N)$ has primitive ideal space 
\[ \text{Prim}\, C^*_{\pi_\tau}(N) = \{ \ker(\sigma_x): x \in F'\}.\]
As in Proposition \ref{prop:Greenimp}, we map Prim $C^*_{\pi_\tau}(N) $ onto $G/H$ by sending $\ker(\sigma_x)$ to $gH$ if $gN = xN$.  This map is continuous and onto, since points are closed in the finite set Prim $C^*_{\pi_\tau}(N)$, and $G$ acts transitively.  Thus, once again, \cite[Theorem 2.13(i)]{Green80} combines with  \cite[Proposition 5.1]{Packer89} to give 
\[ C^*_{\pi_\phi}(G) \cong C^*_{\pi_\tau}(N) \rtimes_{\alpha, \omega} H/N \otimes M_{G/H} \cong C^*_{\pi_\phi|_H}(H) \otimes M_{G/H}. \qedhere\]
\end{proof}

\section{Strong quasidiagonality of virtually nilpotent groups} \label{sec:strongqd}
In this section, we combine the ideas from Section \ref{sec:GrelN} with techniques from von Neumann algebra theory to prove (Corollary \ref{cor:vnilsqd}) that $C^*(G)$ is strongly quasidiagonal for any virtually nilpotent group $G$.  
To streamline the notation, we write $\pi_\tau(G)'$ (respectively $\pi_\tau(G)''$) for the commutant (respectively double commutant) of $C^*_{\pi_\tau}(G)$. 

\begin{lemma} \label{lem:factorsum} Let $G,N$ and $\tau$ be as in Definition \ref{def:tau}.  Then there is an $n\in \N$  and a finite factor $M$ such that
\begin{equation*}
\pi_\tau(N)''\cong \ell^\infty(\{1, \ldots, n \})\otimes M
\end{equation*}
\end{lemma}
\begin{proof} Observe first that the trace $\tau$ extends to a faithful trace on $\pi_\tau(N)''$; consequently, $\pi_\tau(N)''$ is finite. By Lemmas \ref{lem:directsum} and \ref{lem:Jkerpi}, $C^*_{\pi_\tau}(N)$ is a direct sum of simple C*-algebras.   Since each summand is the quotient of a maximal  ideal, it follows that $\pi_\tau$
restricted to each summand is a factor representation by Theorem \ref{thm:T1}. From Equation (\ref{eq: sigmaxdef}) we deduce that each factor summand is isomorphic to each other.
\end{proof}

The main ideas necessary for the proof of Corollary \ref{cor:vnilsqd} are contained in Proposition \ref{prop:fdcenter} below.  Before beginning the proof, we isolate two easy calculations.  The first  is  Schur's Lemma dressed up in our language. 
\begin{lemma} \label{lem:Schur} Let $B\subseteq A$ be C*-algebras and let $\alpha$ be an automorphism of $B.$  Suppose that $a\in A$ satisfies $ba=a\alpha(b)$ for all $b\in B.$  Then $aa^*\in B'\cap A.$
\end{lemma}
\begin{proof} Let $b\in B.$ Then
\begin{equation*}
baa^*=a\alpha(b)a^*=a[a\alpha(b^*)]^*=a[b^*a]^*=aa^*b. \qedhere
\end{equation*}
\end{proof}
\begin{lemma} \label{lem:unitarysum} Let $A$ be a unital C*-algebra and $u\in A$ be unitary.  Suppose there is a scalar $t\in \C$ and a unitary $w\in A$ such that $1+u=tw.$ There is a $\lambda\in \T$ such that either $u=\lambda\cdot 1_A$ or the spectrum of $u$ is $\{ \lambda,\overline{\lambda} \}.$ 
\end{lemma}
\begin{proof} Suppose that $\lambda,\mu$ are in the spectrum of $u.$  Then $1+\lambda$ and $1+\mu$ are in the spectrum of $tw.$  Since $w$ is unitary we have $|1+\lambda|=|1+\mu|.$  Since $|\mu|=|\lambda|=1$, either $\mu=\lambda$ or $\mu=\overline{\lambda}.$
\end{proof}

\begin{proposition} \label{prop:fdcenter} Let $G,N,\tau$ and $\phi$ be as in Definition \ref{def:tau}.  The center of $\pi_\phi(G)''$ is finite-dimensional and $C^*(G)/J_G\cong C^*_{\pi_\phi}(G)$ is a direct sum of simple C*-algebras.
\end{proposition}
\begin{proof}  We first show that $Z(\pi_{\tau_{G/N}\otimes \tau}(G)'')$ is finite dimensional.
The embedding of Equation (\ref{eq:Nembed}) descends to an embedding of $C^*_{\pi_{\tau_{G/N}\otimes \tau}}(N)$ into $C^*_{\pi_{\tau_{G/N}\otimes \tau}}(G).$  Explicitly, for each $h\in N$ we have
\begin{equation*}
\pi_{\tau_{G/N}\otimes \tau}(h)=\sum_{x\in F} e_{xN,xN}\otimes \pi_\tau(x^{-1}hx)\in \pi_{\tau_{G/N}\otimes \tau}(M_{G/N}\otimes C^*(N))\cong M_{G/N}\otimes C^*_{\pi_\tau}(N).
\end{equation*}  
For notational convenience we will write $C=(M_{G/N}\otimes \pi_\tau(N)'')\cap \pi_{\tau_{G/N}\otimes\tau}(N)'.$  It is clear that $Z(\pi_{\tau_{G/N}\otimes \tau}(G)'')\subseteq C.$  We will  prove that $C$ is finite dimensional, from which the finite dimensionality of $Z(\pi_{\tau_{G/N}\otimes \tau}(G)'')$ will follow.

Let $a\in C.$ For $x,y\in F$ write $a_{x,y}\in \pi_\tau(N)''$ for the $(xN,yN)$ matrix entry of $a.$ By the definition of $C$ we have $a_{x,x}\in Z(\pi_\tau(N)'')$ for all $x\in F, $ and in fact 
$\pi_\tau(xhx^{-1})a_{x,y}=a_{x,y}\pi_\tau(yhy^{-1})$ for all $h\in N$ and $x,y\in F.$ Therefore by Lemma \ref{lem:Schur} we have 
 \begin{equation}
 a_{x,y}a_{x,y}^*\in Z(\pi_\tau(N)''). \label{eq:centralcorners}
\end{equation}
Let $M$ and $n$ be as in Lemma \ref{lem:factorsum}. Let $e_1,...,e_n\in  \ell^\infty(\{1,...,n\})$ be mutually orthogonal minimal central projections that sum to the identity.

Fix $x,y\in F.$  We claim that there are unitaries $u_1,...,u_n\in M$ such that 
\begin{equation}
\{a_{x,y}: a\in C \}\subseteq \textup{span}\{ e_i\otimes u_i  : 1\leq i\leq n \}. \label{eq:cornerfd}
\end{equation}
If $x=y$, then $a_{x,x}\in Z(\pi_\tau(N)'')$ so we may take $u_i=1_M$ for all $1\leq i\leq n.$ 

Suppose now that $x\neq y.$  Let $1\leq i\leq n.$  If  $(e_i\otimes 1_M)a_{x,y}=0$ for all $a\in C$, then set $u_i=1_M.$  Now suppose there is some $a\in C$ so $(e_i\otimes 1_M)a_{x,y}\neq0.$  By Equation (\ref{eq:centralcorners}) we have $(e_i\otimes 1_M)a_{x,y}a_{x,y}^*=r^2e_i\otimes 1_M$ for some real number $r>0.$   Since $M\cong e_i\otimes M$ is finite, there is a unitary $u_i\in M$ so 
$(e_i\otimes 1_M)a_{x,y}=ru_i.$

Now let $b\in C.$  We will show that $(e_i \otimes 1)b_{x,y}$ is a scalar multiple of $u_i.$  Without loss of generality suppose that $(e_i \otimes 1)b_{x,y}\neq0.$ By the same argument as above, there is a unitary $v_i\in M$ and a real number $s>0$ such that $(e_i \otimes 1)b_{x,y}=sv_i.$

Since $C$ is a vector space, we have $c=r^{-1}a+s^{-1}b\in C.$  By the same argument as above applied to $c$, there is a real number $t\in \R$ such that $(e_i\otimes 1_M)c_{x,y}=tw$ for some unitary $w\in M.$    Then $u_i+v_i=tw.$  

By way of contradiction, suppose that $u_i$ and $v_i$ are not parallel.  Then
by Lemma \ref{lem:unitarysum}, there is a $\lambda\in \T\setminus\{ \pm1 \}$ such that  the spectrum of $u_i^*v_i$ is $\{ \lambda,\overline{\lambda} \}.$  Then the spectrum of $iu_i^*v_i$ is $\{ i\lambda,i\overline{\lambda} \}.$ But
by the same argument as above applied to the element $r^{-1}a+is^{-1}b\in C$, the spectrum of $iu_i^*v_i$ must be a two element self-adjoint set.  Since $\overline{i\lambda}\neq i\overline{\lambda}$ this is a contradiction.
This establishes Equation (\ref{eq:cornerfd}), from which  the finite dimensionality of $C$ and $Z(\pi_{\tau_{G/N}\otimes \tau}(G)'')$ follow.

We now argue that $Z(\pi_\phi(G)'')$ is also finite dimensional.  Since $\tau_{G/N}\otimes \tau|_G=\phi$, the uniqueness of the GNS  representation implies that $\pi_{\tau_{G/N}\otimes \tau}|_G$ (when restricted to $L^2(C^*(G),\tau_{G/N}\otimes \tau)$) must be unitarily equivalent to 
$(\pi_\phi, L^2(C^*(G),\phi))$. Let $P:L^2(M_{G/N}\otimes C^*(N), \tau_{G/N}\otimes \tau)\rightarrow L^2(C^*(G),\tau_{G/N}\otimes \tau)$ be the orthogonal projection.  Since $G$ leaves    $L^2(C^*(G),\tau_{G/N}\otimes \tau)$ invariant we have $P\in \pi_{\tau_{G/N}\otimes \tau}(G)'$. Hence
\begin{equation*}
P\pi_{\tau_{G/N}\otimes \tau}(G)''\cong \pi_\phi(G)''.
\end{equation*}
Since von Neumann quotients of von Neumann algebras are cutdowns by central projections (see, for example \cite[Theorem III.2.7]{Takesaki02}), it follows that the center of $\pi_\phi(G)''$ is also finite dimensional.

It remains to show that $C^*_{\pi_\phi}(G)$ is a direct sum of simple $C^*$-algebras.  Our arguments above  imply  that $\pi_\phi(G)''\cong \mathcal{M}_1\oplus \cdots \oplus\mathcal{M}_k$ where each $\mathcal{M}_i$ is a finite factor. Writing $p_i\in \pi_\phi(G)''$ for the central projection onto the $i$th summand, it follows that $p_i\pi_\phi$ determines a factor representation of $G.$  Therefore $\ker(p_i\pi_\phi)$  is maximal by Theorem \ref{thm:T1}.  Choose a subset $A\subseteq \{ 1,\ldots,k \}$ such that 
\begin{equation*}
C^*(G)/\cap_{i\in A}\ker(p_i\pi_\phi)\cong C^*(G)/\cap_{i=1}^k\ker(p_i\pi_\phi)
\end{equation*}
and $i,j\in A$ with $i\neq j$ implies that $\ker(p_i\pi_\phi)\neq \ker(p_j\pi_\phi).$  By the Chinese remainder theorem, 
\begin{equation*}
C^*_{\pi_\phi}(G) \cong 
C^*(G)/\cap_{i\in A}\ker(p_i\pi_\phi)\cong \bigoplus_{i\in A}C^*(G)/\ker(p_i\pi_\phi)
\end{equation*}
is a direct sum of simple C*-algebras.
\end{proof}

\begin{theorem} \label{thm:qdquotient}  Let $G$ be virtually nilpotent and $\pi$ an irreducible representation of $G.$  Then $C_\pi^*(G)$ is quasidiagonal.
\end{theorem}
\begin{proof} Let $N\nor G$ be finite index and nilpotent.  Let $\tau$ and $\phi$ be as in Definitions \ref{def:setup} and \ref{def:tau}. 
Recall from Equation (\ref{eq:JGphi}) that  $J_G=\{ x\in C^*(G):\phi(x^*x) =0\}\subseteq \ker(\pi).$ By Proposition \ref{prop:fdcenter}, $C^*_{\pi_\phi}(G)$ is a direct sum of simple C*-algebras. Hence $C_\pi^*(G)$ must appear as a direct summand of $C^*_{\pi_\phi}(G)$.  Moreover,  \cite{Eckhardt14} proves that  $M_{G/N}\otimes C^*(N)$ is strongly quasidiagonal whenever $N$ is nilpotent.  Since $C^*_{\pi_\phi}(G)$ embeds into $M_{G/N}\otimes C^*_{\pi_\tau}(N)$, the C*-algebra $C^*_{\pi_\phi}(G)$ is quasidiagonal and therefore so is $C_\pi^*(G).$
\end{proof}
\begin{corollary}\label{cor:vnilsqd} Let $G$ be an inductive limit of virtually nilpotent groups.  Then $C^*(G)$ is strongly quasidiagonal.
\end{corollary}
\begin{proof} Since strong quasidiagonality is preserved by injective inductive limits we may suppose that $G$ is virtually nilpotent.  By \cite[Lemma 1.5]{Eckhardt14} we only need to show that every primitive quotient of $C^*(G)$ is quasidiagonal.  This is exactly the content of Theorem \ref{thm:qdquotient}.
\end{proof}
\begin{remark} Our method of proof for Theorem \ref{thm:qdquotient} provides a shorter, cleaner route to the results of \cite{Eckhardt14} on strong quasidiagonality of nilpotent groups. Indeed, one may replace the long and technical \cite[Lemma 3.1]{Eckhardt14} with Proposition \ref{prop:fdcenter} above.  
\end{remark}

\section{Finite decomposition rank of finitely-generated  nilpotent group C*-algebras} \label{sec:finitedrnilpotent}

In this section we extend the main result of \cite{Eckhardt14b} to show that $C^*(G)$ has finite decomposition rank whenever $G$ is
finitely-generated and nilpotent.  Our proof will proceed by induction on the Hirsch length of $G$.

Recall  that finitely generated virtually nilpotent groups $G$  are virtually \emph{polycyclic}, that is, $G$ contains a finite index subgroup $N$ that admits a normal series
\begin{equation*} 
\{ e \}=N_n \trianglelefteq N_{n-1}\trianglelefteq \cdots \trianglelefteq N_1\trianglelefteq N_0=N
\end{equation*}
where $N_{i}/N_{i+1}$ is cyclic for each $i=0,\ldots,n-1.$ The {\bf Hirsch length} $h(G)$ of $G$ is the number of times the quotients $N_i/N_{i+1}$ are infinite.  For any normal subgroup $H \nor G$ of a virtually polycyclic group $G$, both $H$ and $G/H$ are virtually polycyclic and  
\[ h(G) = h(G/H) + h(H).\]
We refer the reader to Segal's book \cite{Segal83} for proofs of the above facts and for more information on polycyclic groups and their Hirsch length.
\begin{theorem} \label{thm:drnil}
  Let $G$ be a finitely generated nilpotent group.  Then the decomposition rank of $C^*(G)$ is at most $2\cdot h(G)!  - 1$, where $h(G)$ is the Hirsch length of $G$.
\end{theorem}
\begin{proof} We proceed by induction on  $h(G)$.  If $h(G) = 0$, then $G$ is finite, and hence $C^*(G)$ is finite-dimensional and has decomposition rank
  $0$.  So we suppose $h(G) > 0$.  We decompose $C^*(G)$ as a continuous field over $\widehat{Z(G)}.$ In the notation of Section \ref{sec:ctsfield}, the   fiber
  over $ \gamma \in \widehat{Z(G)}$ is $C^*_{\pi_{\overline \gamma}}(G)$, where $\overline\gamma$ is the trivial extension of $\gamma$ (cf.~Section \ref{sec:trivialext}).  We claim that the decomposition rank of each fiber is at most $2(h(G)-1)! - 1$.

Suppose first that  $\ker{\gamma}$ is finite.  Arguing as in~\cite[Lemma~3.3]{Eckhardt14b}, we may write $\overline{\gamma}$ as a
  finite convex combination of extreme traces $\omega_1,\ldots,\omega_n$ on $G$, and hence we may view $C^*_{\pi_{\overline \gamma}}(G)$ as a finite direct sum of C*-algebras where each summand is isomorphic to some $C^*_{\pi_{\omega_i}}(G).$ Each C*-algebra $C^*_{\pi_{\omega_i}}(G)$ is simple by Theorem \ref{thm:T1}. Therefore by Theorem~\ref{thm:nildr1}, $\dr(C^*_{\pi_{\omega_i}}(G))\leq1.$
 By (\ref{eq:drsum}) we have $\dr(C^*_{\pi_{\overline \gamma}}(G))\leq1.$
 
Now suppose that $\ker{\gamma}$ is infinite. Then $h(\ker{\gamma}) > 0$, so $h(G / \ker{\gamma}) < h(G)$.  By the induction hypothesis and (\ref{eq:qdr}) we have  $\dr(C^*_{\pi_{\overline \gamma}}(G))\leq \dr(C^*(G / \ker{\gamma}))\leq 2(h(G)-1)! - 1.$ 

 Now that we have a bound on $\dr(C^*_{\pi_{\overline \gamma}}(G))$, we may apply \cite[Lemma~3.1]{Carrion11} to obtain
  \[
    \dr(C^*(G)) \le 2(\dim{\widehat{Z(G)}} + 1) (h(G)-1)! - 1.
  \]
  Note that $\dim(\widehat{Z(G)}) = h(Z(G)) \le h(G)$; but in fact if equality is achieved then $G$ is a finite extension of $Z(G)$, in which case we have
  $\dr(C^*(G)) = h(G) \le 2h(G)! - 1$.  So we may assume $\dim(\widehat{Z(G)}) < h(G)$, which gives
  \[
    \dr(C^*(G)) \le 2h(G)! - 1
  \]
  as required.
  
\end{proof}

\section{Finite decomposition rank of finitely generated, virtually nilpotent groups} \label{sec:finitedr}

This section combines our results on quasidiagonality (Corollary \ref{cor:vnilsqd}) and decomposition rank for nilpotent group C*-algebras (Theorem \ref{thm:drnil}) to prove our main result in Theorem \ref{thm:main} below. The first step to proving that $C^*(G)$ has finite decomposition rank whenever $G$ is finitely generated and virtually nilpotent is to bound $\dr (C^*_\pi(G))$ for each irreducible representation $\pi$ of $G$.  Theorem \ref{thm:mainirr}, which computes this bound, relies on a general result (Theorem \ref{thm:onlyamother}) about maximal ideals in group C*-algebras.  Finally, Theorem \ref{thm:main} establishes that the upper bound we obtained for $\dr(C^*(G))$ in Theorem \ref{thm:drnil} for nilpotent $G$ also bounds $\dr(C^*(G))$ when $G$  is virtually nilpotent.

{Although $G$ need not be amenable in the following Theorem, its proof holds verbatim if we replace $C^*(N)$ and $C^*(G)$ with $C^*_r(N)$ and $C^*_r(G)$.}

\begin{theorem} \label{thm:onlyamother} Let $N\nor G$ be discrete groups and $|G/N|<\infty.$  Suppose that $I\subseteq C^*(N)$ is a maximal ideal
such that $u_x I u_{x^{-1}}=I$ for all $x\in G$.  Suppose that $C^*(N)/I$ is a finite C*-algebra. Let $I_G\subseteq C^*(G)$ be the ideal of $C^*(G)$ generated by $I.$  If $I_G$ is not a maximal 
ideal of $C^*(G)$, then there exists $x\in G\setminus N$ such that $a+I\mapsto u_xa u_{x^{-1}}+I$ is an inner automorphism of $C^*(N)/I.$
\end{theorem}
\begin{proof}    Let $e\in F\subseteq G$ be a complete set of $G/N$-coset representatives. We first prove that there is an $a\in C^*(N)\setminus I$ and an $x\in F$ such that 
\begin{equation}
u_h a-a u_{xhx^{-1}}\in I \textup{ for all }h\in N. \label{eq:xtwist}
\end{equation}
We prove Equation (\ref{eq:xtwist}) by the same method as \cite[Lemma 3.3]{Eckhardt14}. 
For each $z\in C^*(G)$ and each $x \in F$, there is a unique $z_x\in C^*(N)$ such that
 \begin{equation*}
 z=\sum_{x\in F} z_x u_x.
 \end{equation*}
Define the \emph{support} of $z$ as $\supp(z)=\{ x\in F: z_x\not\in I \}.$ By definition, 
\begin{equation*}
I_G=\left\{ \sum_{x\in F}a_x u_x: a_x\in I  \right\},
\end{equation*}
so $I_G$ consists of precisely those elements of $C^*(G)$ with empty support.

Let $I_G\subseteq I_{max}\subseteq C^*(G)$ be a maximal ideal.  By assumption we have $I_G\neq I_{max}.$
Let $z\in I_{max}\setminus I_G$ be such that $|\supp(z)|$ is minimal. Since $z\not\in I_G$ it has non-empty support.
We claim that there is an  $x\in F$ such that the pair $z_x=a$ and $x\in F$ satisfies Equation (\ref{eq:xtwist}).
By possibly multiplying $z$ by an appropriate $u_{x^{-1}}$ we may -- without changing $|\supp(z)|$ -- suppose that $z_e\not\in I.$  

Since $I\subseteq C^*(N)$ is a maximal ideal and $z_e\not\in I$, there are $\alpha_i,\beta_i\in C^*(N)$ so $\sum_i \alpha_iz_e\beta_i+I=1+I.$ In other words, we can replace $z$ by 
\[ z' := \sum_i \sum_{x\in F} \alpha_i z_x \beta_i u_x + (1 - \sum_i \alpha_i z_e \beta_i)\]
without enlarging $\text{supp}(z)$.  Moreover,  $z'_e = 1$.    Hence, we suppose we have chosen  $z $ such that $z_e=1$. 

Since $z\in I_{max}$ and $z_e+I=1+I$ we have $z\neq z_e.$  In other words, there must be some $x\in F\setminus\{ e \}$ such that $z_x\not\in I.$
Suppose that there is some $h \in N$ such that 
$ u_h z_x - z_x u_{x hx^{-1}} \not \in I$.
Then 
\[u_hz-zu_h = \sum_{y\in F} \left(u_h z_y - z_y u_{yhy^{-1}} \right) u_y \in I_{max}\setminus I_G,\]
 but whenever $z_y \in I$ we also have $u_h z_y - z_y u_{yhy^{-1}} \in I$.  Combined with our assumption that $z_e =1$, this implies that $\supp( u_hz-z u_h)\subseteq \supp(z)\setminus\{e \}$, contradicting the minimality of the support of $z.$
 Hence we have established Equation (\ref{eq:xtwist}).
\\\\
We now return to the main proof.  Let $x,a$ satisfy Equation (\ref{eq:xtwist}).   For all $h\in N$ we then have
\begin{equation}
u_ha+I=a u_{xhx^{-1}}+I. \label{eq:intertwiner}
\end{equation}
By Lemma \ref{lem:Schur},  $aa^*+I$ is in the center of $C^*(N)/I.$  But $C^*(N)/I$ is simple, so $a a^* + I$ must be a scalar multiple of the identity. 
 Since $a\not\in I$, our hypothesis that  $C^*(N)/I$ is a finite C*-algebra then implies that $a+I$ is a nonzero scalar multiple of a unitary.  Therefore, $u:=\V a+I \V^{-1}a$ represents a unitary element in $C^*(N)/I$, and by Equation (\ref{eq:intertwiner}) we have $u^* u_hu+I=u_{xhx^{-1}}+I.$
  \end{proof}

\begin{theorem} \label{thm:mainirr} Let $G$ be a finitely generated, virtually nilpotent group and $\pi$ an irreducible representation of $G$.  Then $\dr(C_\pi^*(G))\leq1.$ 
\end{theorem}
\begin{proof}  By Theorem \ref{thm:drnil} there is a finite index normal subgroup $H$ of $G$ such that $\dr(C^*(H))<\infty.$ Therefore every quotient of $C^*(H)$ also has finite decomposition rank (and therefore finite nuclear dimension) by (\ref{eq:qdr}). From Theorem \ref{thm:qdquotient} and Theorem \ref{thm:T1}, it follows that every primitive quotient of $C^*(H)$ satisfies the conditions of Theorem \ref{thm:manyequivalence}.  
In other words, $G$ has  a finite index subgroup $H$  such that $\dr(C^*_\sigma(H)) \leq 1$ for all irreducible representations $\sigma$ of $H$. 

We proceed by induction on $|G/N|$ where $N$ is a finite index normal subgroup of $G$ with $\dr(C^*_\sigma(N)) \leq 1$ for all irreducible representations $\sigma$ of $N.$\footnote{We emphasize that $N$ is not necessarily nilpotent.}  
Since we proceed by induction, we may suppose that $G/N$ is a simple group.  Let $\pi$ be an irreducible representation of $G.$  From $\pi$ we obtain the representation $\sigma$ of $N$, the ideals $J\subseteq C^*(N)$ and $J_G\subseteq C^*(G)$ and the traces $\tau$ on $N$ and $\phi$ on $G$ as in Definitions \ref{def:setup} and \ref{def:tau}.

If $\Stab_G(\ker(\sigma))=H\neq G$, then by  Corollary \ref{cor:Greenimprep} we have $C^*_{\pi_\phi}(G)\cong M_{G/H}\otimes C_{\pi_\phi}^*(H).$  By Proposition \ref{prop:fdcenter}, $C^*_{\pi_\phi}(G)$ is a direct sum of simple C*-algebras.  Therefore $C^*_{\pi_\phi}(H)$ is also a direct sum of simple C*-algebras.  Since $|H/N|<|G/N|$,  by our induction hypothesis    we obtain $\dr(C^*_{\pi_\phi}(H))\leq1.$  Therefore 
 $\dr(C^*_{\pi_\phi}(G))=\dr(C_{\pi_\phi|_H}^*(H))\leq 1$ by (\ref{eq:drstable}).   
Since $C^*_\pi(G)$ is a quotient of $C^*_{\pi_\phi}(G)$ we have  $\dr(C^*_\pi(G))\leq \dr(C^*_{\pi_\phi}(G))\leq1$ by (\ref{eq:qdr}).

Therefore suppose that  $\Stab_G(\ker(\sigma))=G,$ so that  $J=\ker(\sigma).$ Then  $C^*_{\pi_\tau}(N)\cong C^*(N)/J$ is a simple C*-algebra. By Theorems \ref{thm:T1} and \ref{thm:qdquotient}, $C^*_{\pi_\tau}(N)$ is quasidiagonal with a unique trace.  If $C^*_{\pi_\tau}(N)$ is finite dimensional, then so is $C^*_\pi(G)$, hence $\dr(C^*_\pi(G))=0.$  Therefore suppose that $C^*_{\pi_\tau}(N)$ is infinite dimensional. By our induction hypothesis and Theorem \ref{thm:manyequivalence} it follows that $C^*_{\pi_\tau}(N)$
is $\mathcal{Z}$-stable.

Since $\alpha_s(J)=J$ for every $s\in G/N$ it follows that the twisted action $(\alpha,\omega)$ of $G/N$ on $C^*(N)$ from Section \ref{sec:ctsfield} descends to a twisted action on $C^*_{\pi_\tau}(N).$  For convenience we also write $(\alpha,\omega)$ for this twisted action. We then have $C^*_{\pi_\phi}(G)\cong C^*_{\pi_\tau}(N)\rtimes_{\alpha,\omega} G/N.$

The twisted action of $G/N$ on $C^*_{\pi_\tau}(N)$  defines a homomorphism from $G/N$ into $\textup{Out}(\pi_\tau(N)'').$  Since $G/N$ is simple, either every non-trivial $s\in G/N$ determines an outer automorphism of $\pi_\tau(N)''$  or none of them do.  We consider each case separately.
\\\\
\textbf{Strongly outer case:} Suppose first that the twisted action of $G/N$ is strongly outer.  It follows from  Matui and Sato's theorem \cite[Corollary 4.10]{Matui14a} that $C^*_{\pi_\tau}(N)\rtimes_{\alpha,\omega} G/N$ is $\mathcal{Z}$-stable, and therefore $\dr(C^*_{\pi_\phi}(G))=\dr(C^*_{\pi_\tau}(N)\rtimes_{\alpha,\omega} G/N)\leq 1$ by Theorem \ref{thm:manyequivalence}. Since $C^*_\pi(G)$ is a quotient of $C^*_{\pi_\phi}(G)$ we have $\dr(C^*_\pi(G))\leq \dr(C^*_{\pi_\phi}(G))\leq 1.$
\\\\
\textbf{Weakly inner case:} Now suppose that for every $s\in G/N$, the induced automorphism $\alpha_s$ determines an inner automorphism of $\pi_\tau(N)''.$  We will show that every $\alpha_s$ actually determines an inner automorphism of $C^*_{\pi_\tau}(N).$ 

Let $e\in F\subseteq G$ be a complete set of coset representatives of $G/N.$ Let $s\in F\setminus \{ e\}$.   Let $n$ be the order of $sN\in G/N$, i.e. $n>0$ is the least integer such that $s^n\in N.$ Let $H\leq G$ be the subgroup of $G$ generated by $s$ and $N.$  Then $N\nor H$ and $H/N\cong \Z/n\Z.$    Let $J_H$ be the ideal of $C^*(H)$ generated by $J.$  

If $J_H$ is not a maximal ideal of $C^*(H)$, then there is a $1\leq k<n$ such that $\alpha_{s^k}$ is an inner automorphism of $C^*_{\pi_\tau}(N)$ by Theorem \ref{thm:onlyamother}.  Since $G/N$ is simple, it then follows that  $\alpha_t$ defines an inner automorphism of $C^*_{\pi_\tau}(N)$ for every $t\in G/N.$ Therefore  $C^*_{\pi_\phi}(G)\cong C^*_{\pi_\tau}(N)\rtimes_{\alpha,\omega}G/N$ is $\mathcal{Z}$-stable by Proposition \ref{prop:Zstablefixedpoint}. 
By Proposition \ref{prop:fdcenter}, $C^*_{\pi_\phi}(G)$ is a direct sum of simple C*-algebras.  Since $\ker(\pi_\phi)\subseteq \ker(\pi)$ it then follows that $C^*_\pi(G)$ appears as a direct summand of $C^*_{\pi_\phi}(G)$, so $C^*_\pi(G)$ is also $\mathcal{Z}$-stable.  We then have $\dr(C^*_\pi(G))\leq 1$ by Theorem \ref{thm:manyequivalence}.

Suppose that $J_H$ is a maximal ideal of $C^*(H).$ We will show that this leads to a contradiction.  We have $C^*(H)/J_H\cong C^*_{\pi_\tau}(N)\rtimes_{\alpha,\omega} H/N$ where $(\alpha,\omega)$ is the twisted action of $G/N$ restricted to $H/N.$ The isomorphism class of $C^*_{\pi_\tau}(N)\rtimes_{\alpha,\omega} H/N$ is independent of which coset representatives of $G/N$ were chosen (see Section \ref{sec:ctsfield}).  Therefore for each $1\leq k<n$ suppose we chose liftings $c(k+n\Z)=s^k\in H\leq G.$ Whenever liftings are chosen in this manner,  we can explicitly write out the values of the cocycle $\omega$ from (\ref{eq:omegadef}). If $0\leq i,j< n$, then
\begin{equation} \label{eq:twistdefn}
\omega(i+n\Z, j+n\Z)=\left\{ \begin{array}{ll} 1_{H_{\pi_\tau}} & \textup{ if }i+j<n  \\
                                                                                       \pi_\tau(s^n) & \textup{ if }i+j\geq n  \end{array}\right..
\end{equation}
Let $W\in \pi_\tau(N)''$ be a unitary such that $W\pi_\tau(a)W^*=\pi_\tau(\alpha_s(a))=\pi_\tau(\lambda_sa\lambda_{s^{-1}})$ for all $a\in C^*_{\pi_\tau}(N).$ Therefore $\pi_\tau(s^n)^*W^n\in \pi_\tau(N)'\cap \pi_\tau(N)''.$  Since $\pi_\tau$ is a factor representation of $N$ there is a scalar $e^{i\theta}\in \T$ such that $W^n=e^{i\theta}\pi_\tau(s^n).$  Replacing $W$ with $e^{-i\theta/n}W$ we may assume that $W^n=\pi_\tau(s^n)$.

Now define a covariant representation  of $(C^*_{\pi_\tau}(N),H/N,\alpha,\omega)$ on $H_{\pi_\tau}$ which is the identity on  $C^*_{\pi_\tau}(N)$ and sends $k\in \Z/n\Z\cong H/N$ to $W^k.$ Equation (\ref{eq:twistdefn}) and the fact that $W^n=\pi_\tau(s^n)$ tell us that, indeed, 
\[ \alpha_{s^k} = \text{Ad} W^k \quad \text{ and } \quad W^k W^j = \omega(k, j) W^{k+j},\]
 so $(id, W)$ satisfies the definition of a  covariant representation from Definition \ref{def:twisted-xprod}.  Write $C^*(\pi_\tau(N), W)$ for the quotient of $C^*_{\pi_\tau}(N) \rtimes_{\alpha, \omega} H/N$ generated by this covariant representation.

Since $J_H$ is a maximal ideal, the C*-algebra 
\[ C^*_{\pi_\tau}(N) \rtimes_{\alpha, \omega} H/N\cong C^*(H)/J_H\cong C^*_{\pi_\phi}(H)\]
 is simple. Therefore our covariant representation defines an isomorphism  $\Theta:C^*_{\pi_\phi}(H)\rightarrow C^*(\pi_\tau(N),W)$ such that $\Theta(\pi_\phi(h))=\pi_\tau(h)$ for all $h\in N$ and $\Theta(\pi_\phi(s))=W.$  Theorem \ref{thm:T1} implies that $C^*_{\pi_\phi}(H)$ (and hence $C^*(\pi_\tau(N), W)$) has a unique trace.  
Furthermore, since $C^*(\pi_\tau(N),W)\subseteq \pi_\tau(N)''$ we have $\tau\circ \Theta(a)=\phi(a)$ for all $a\in C^*_{\pi_\phi}(H).$

From Equation (\ref{eq:phitau}) we have  $\phi(\pi_\phi(sh))=0$
for all $h\in N.$ Therefore 
\begin{equation*}
\tau(W\pi_\tau(h))=\Theta(\pi_\phi(sh))=0\quad \textup{ for all }h\in N.
\end{equation*}
Since any C*-algebra is weakly dense in its double commutant,  it follows that $\tau(Wb)=0$ for all $b\in \pi_\tau(N)''.$ 
But this is impossible: since $W\in \pi_\tau(N)''$ is unitary, $\tau(WW^*)=1\neq0.$
\end{proof}

\begin{corollary}
\label{cor:fin-nuc-dim}
For any irreducible representation $\pi$ of a finitely generated, virtually nilpotent group $G$, if $C^*_\pi(G)$ is not AF then $\dr(C^*_\pi(G) ) = 1.$  Moreover, in this case $C^*_\pi(G)$ has nuclear dimension 1. 
\end{corollary}
\begin{proof}
It is well known (cf.~\cite[Remarks 2.2]{Winter10}) that $\text{dim}_{\text{nuc}}(A) \leq \dr(A)$ for any C*-algebra $A$, and that 
\[ \text{dim}_{\text{nuc}}(A)=0 \Leftrightarrow \dr(A) = 0 \Leftrightarrow A \text{ is AF.} \qedhere\]
\end{proof}

\begin{theorem} \label{thm:main} Let $G$ be a finitely generated, virtually nilpotent group.  Then $\dr(C^*(G))\leq 2\cdot h(G)!-1$, where $h(G)$ is the Hirsch length of $G.$
\end{theorem}
\begin{proof}  We proceed by induction on  $h(G).$  If $h(G)=0$, then $G$ is finite.  Hence $C^*(G)$ is finite dimensional and $\dr(C^*(G))=0.$

Let $N\nor G$ be nilpotent and finite index. Since every finitely generated nilpotent group contains a torsion free, finite index subgroup \cite{Segal83} we may assume that $N$ is torsion free. 

Decompose $C^*(G)$ as a continuous field over $\widetilde Z$ as in Section \ref{sec:ctsfield}. Below we use the notation of Section \ref{sec:ctsfield}. Let $\gamma\in \widehat{Z(N)}$ and write $H$ for the stabilizer of $\gamma$.  We will show that the decomposition rank of the fiber $C^*(G)_{[\gamma]}$ is at most $2(h(G)-1)!-1.$

Suppose first that $\ker(\gamma)\leq Z(N)$ is infinite.  Then $\ker(\gamma)\nor H,$ and the representation $\pi_{\overline{\gamma}}$ of $H$ from Proposition \ref{prop:Greenimp} descends to a representation of $H/\ker(\gamma)$.  Consequently, 
 Proposition \ref{prop:Greenimp} 
implies that $C^*(G)_{[\gamma]}$ is stably isomorphic to a quotient of the group $C^*$-algebra $C^*(H/\ker(\gamma)).$  Since $\ker(\gamma)$ is infinite, we have $h(H/\ker(\gamma))\leq h(G)-1.$  Therefore, by our induction hypothesis and the fact that decomposition rank is insensitive to stable isomorphism (\ref{eq:drstable}) we have
\begin{equation*}
\dr(C^*(G)_{[\gamma]}) = \dr(C^*(H/\ker(\gamma)))\leq 2(h(G)-1)!-1.
\end{equation*}

Now suppose  that $\ker(\gamma)$ is finite. Since $N$ is torsion free and $\ker(\gamma) \nor N$, the finiteness of $\ker(\gamma)$ implies that $\ker(\gamma)$ is trivial.
 Let $\overline{\gamma}$ be the trivial extension of $\gamma$ to $N.$  Since $N$ is torsion free, by \cite[Lemma 2.3]{Eckhardt14b}, $\overline{\gamma}$ determines an extreme trace on $N.$  Therefore $J:=\ker(\pi_{\overline{\gamma}})$ is a maximal ideal of $C^*(N)$ by Theorem \ref{thm:T1}. 

Although $\overline{\gamma}$ need not induce an extreme trace on $G$,  we can still use the structural analysis of Section \ref{sec:GrelN}. To see this, let $J_G$ be the ideal of $C^*(G)$ generated by $J.$ For any maximal ideal $I_{max}$ of $C^*(G)$ containing $J_G$ we have $I_{max}\cap C^*_{\pi_\tau}(N)=J$ by maximality of $J;$  thus, the ideals $J$ and $J_G$ match up with the ideals $J$ and $J_G$ of Definition \ref{def:tau}.
Consequently, Proposition \ref{prop:Greenimp} and Corollary \ref{cor:Greenimprep} imply that $C^*(G)_{[\gamma]}\cong C^*(G)/J_G$, and the latter is isomorphic to a direct sum of maximal quotients of $C^*(G)$  by Proposition \ref{prop:fdcenter}. Hence $\dr(C^*(G)_{[\gamma]})\leq1$ by Theorem \ref{thm:mainirr} and (\ref{eq:drsum}).
We have therefore shown that $\dr(C^*(G)_{[\gamma]})\leq 2(h(G)-1)!-1$ for all $\gamma\in \widehat{Z(N)}.$

Since $\widetilde Z$ is a quotient of $\widehat{Z(N)}$ by a finite group action we have $\dim(\widetilde{Z})=\dim(\widehat{Z(N)}).$ 
Using our bound on $\dr(C^*(G)_{[\gamma]})$, we  apply \cite[Lemma~3.1]{Carrion11} to obtain
\begin{align*}
    \dr(C^*(G)) & \le 2(\dim{\widetilde Z} + 1) (h(G)-1)! - 1\\
                      & =  2(\dim{\widehat{Z(N)}} + 1) (h(G)-1)! - 1
\end{align*}
  Note that $\dim(\widehat{Z(N)}) = h(Z(N)) \le h(N)=h(G)$; if equality is achieved, then 
  \[C^*(G)_{[\gamma]} \cong \bigoplus_{\eta \in [\gamma]} C^*(N/Z(N), \sigma_\eta) \rtimes_{\tilde \alpha, \tilde \omega} G/N\]
   is finite dimensional for all $\gamma$.   In this case, 
  $\dr(C^*(G)) = h(G) \le 2h(G)! - 1$ by \cite[Lemma~3.1]{Carrion11}.  Finally, if $\dim(\widehat{Z(G)}) < h(G)$, our bound above on $\dr(C^*(G))$ becomes, again,
\begin{equation*}
    \dr(C^*(G)) \le 2h(G)! - 1. \qedhere
\end{equation*}
\end{proof}
\section{Questions and Comments} \label{sec:comments}
We finish this paper with three natural questions.
\begin{question} \label{ques:otherfdr} Are there any finitely generated  groups $G$ which are not virtually nilpotent but which have  $\dr(C^*(G))<\infty$?
\end{question}
 In some sense the only known obstruction to finite decomposition rank is the lack of strong quasidiagonality \cite[Theorem 5.3]{Kirchberg04}. Indeed, in the case of classifiable C*-algebras \cite[Theorem 6.2.(iii)]{Tikuisis17} (strong) quasidiagonality \emph{is} the only obstruction to decomposition rank and marks the dividing line between finite nuclear dimension and finite decomposition rank.
 
In \cite{Eckhardt15} the first author gave examples of polycyclic but not virtually nilpotent groups whose group C*-algebras are  strongly quasidiagonal. These are therefore natural candidates to provide a positive answer to Question \ref{ques:otherfdr}. On the other hand, these groups lack any sort of nice continuous field decomposition, leaving the methods of this paper inapplicable.
\begin{question} \label{ques:uct}
Let $G$ be virtually nilpotent and $\pi$ an irreducible representation of $G.$  Does $C^*_\pi(G)$ satisfy the universal coefficient theorem?
\end{question}
In \cite{Eckhardt16} the first two authors showed that the answer is yes whenever $G$ is nilpotent, by establishing structure theorems for $C^*_\pi(G)$ in this case.  We then deduced the universal coefficient theorem from the permanence properties established by Rosenberg and Schochet  \cite{Rosenberg87} in their initial work on the universal coefficient theorem.

We believe that such an approach will no longer work for virtually nilpotent groups.
The universal coefficient theorem is notoriously difficult to prove in the presence of torsion.  In fact, Barlak and Szab\'o show in \cite[Theorem 4.17]{Barlak17}  that \emph{all} nuclear C*-algebras satisfy the universal coefficient theorem if and only if all crossed products of the form $\mathcal{O}_2\rtimes \Z/p\Z$ satisfy the universal coefficient theorem.  

Torsion can, of course,  be present in nilpotent groups.  In the nilpotent setting, however, \cite{Eckhardt16} shows that all of the torsion can be ``pushed to the bottom"  of the C*-algebra.  {Specifically,  \cite[Theorem 3.4]{Eckhardt16} establishes that when $G$ is nilpotent, $C^*_\pi(G) \cong A \rtimes \Z \cdots \rtimes \Z$, where $A = C_0(X) \rtimes_{\alpha, \omega} F$ is a twisted crossed product of an abelian $C^*$-algebra by a finite group $F$.} For  virtually nilpotent groups, however, we must deal with torsion ``at the top;" the C*-algebras generated by irreducible representations of virtually nilpotent groups may be of the form  $A\rtimes F$, where the C*-algebra $A$ satisfies the universal coefficient theorem and $F$ is a finite group.

We would love to be wrong, but with present tools and especially in light of  Barlak and Szab\'o's result it does not appear that Question \ref{ques:uct} is any  easier than the general universal coefficient question for all nuclear C*-algebras.
\begin{question}  Can the results of this paper be extended to all groups?  Suppose that $N\leq G$ is a finite index subgroup.  Does  finite decomposition rank (resp. nuclear dimension) of $C^*(G)$ imply the same for $C^*(N)$?  Does  finite decomposition rank (resp. nuclear dimension) of $C^*(N)$ imply the same for $C^*(G)?$ 
\end{question}

\bibliographystyle{plain}

\bibliography{mybib}

\end{document}